\documentclass{article}%
\usepackage[a4paper]{geometry}
\usepackage{amsmath}
\usepackage{amsfonts}
\usepackage{amssymb}
\usepackage{graphicx}%
\providecommand{\U}[1]{\protect\rule{.1in}{.1in}}
\newtheorem{theorem}{Theorem}

\newtheorem{lemma}[theorem]{Lemma}

\newtheorem{proposition}[theorem]{Proposition}
\newtheorem{remark}[theorem]{Remark}

\newenvironment{proof}[1][Proof]{\noindent\textbf{#1.} }{\ \rule{0.5em}{0.5em}}
\begin{document}

\title{Semiclassical measures and the Schr\"{o}dinger flow on Riemannian manifolds}
\author{Fabricio Maci\`{a}\thanks{This research has been supported by program
\emph{Juan de la Cierva }(MEC, Spain) and projects MAT2005-05730-C02-02 (MEC,
Spain) and HYKE (E.U. ref. HPRN-CT-2002-00282).}\\Universidad Polit\'{e}cnica de Madrid\\ETSI Navales\\Avda. Arco de la Victoria, s/n. 28040 Madrid Spain\\E-mail address: \texttt{fabricio.macia@upm.es}}
\date{}
\maketitle

\begin{abstract}
In this article we study limits of Wigner distributions (the so-called
semiclassical measures) corresponding to sequences of solutions to the
semiclassical Schr\"{o}dinger equation at times scales $\alpha_{h}$ tending to
infinity as the semiclassical parameter $h$ tends to zero (when $\alpha
_{h}=1/h$ this is equivalent to consider solutions to the non-semiclassical
Schr\"{o}dinger equation). Some general results are presented, among which a
weak version of Egorov's theorem that holds in this setting. A complete
characterization is given for the Euclidean space and Zoll manifolds (that is,
manifolds with periodic geodesic flow) via averaging formulae relating the
semiclassical measures corresponding to the evolution to those of the initial
states. The case of the flat torus is also addressed; it is shown that
non-classical behavior may occur when energy concentrates on resonant
frequencies. Moreover, we present an example showing that the semiclassical
measures associated to a sequence of states no longer determines those of
their evolutions. Finally, some results concerning the equation with a
potential are presented. \medskip

\noindent\textbf{Mathematics Subject Classification:} Primary 81Q20; Secondary
37J35, 37N20, 58J47.

\end{abstract}

\section{Introduction}

The \emph{quantum-classical correspondence principle }roughly states that
quantum systems behave according to classical mechanics in the high-frequency
limit. A particular case that has attracted special attention corresponds to
taking as the underlying classical system the geodesic flow on a complete
Riemannian manifold $\left(  M,g\right)  $. Its quantum counterpart is the
Schr\"{o}dinger flow, \emph{i.e. }the unitary group $e^{ith\Delta/2}$
generated by the Laplace-Beltrami operator $\Delta$ on $L^{2}\left(  M\right)
$. In order to relate its high-frequency properties to the geodesic flow, one
tries to determine the limiting behavior as $h\rightarrow0$ of the position
densities $\left\vert \psi_{h}\left(  t,\cdot\right)  \right\vert ^{2}$
associated to solutions to the Schr\"{o}dinger equation:%
\begin{equation}
ih\partial_{t}\psi_{h}\left(  t,x\right)  +\frac{h^{2}}{2}\Delta\psi
_{h}\left(  t,x\right)  =0\qquad\left(  t,x\right)  \in\mathbf{R}\times M,
\label{SE}%
\end{equation}
issued from a sequence of highly oscillating initial data $\psi_{h}%
|_{t=0}=u_{h}$, whose characteristic lengths of oscillations are of order $h$.
One expects that in this limit the dynamics of $\left\vert \psi_{h}\left(
t,\cdot\right)  \right\vert ^{2}$ are somehow related to the geodesic flow.

Usually, it is preferable to consider instead of $\left\vert \psi_{h}\left(
t,\cdot\right)  \right\vert ^{2}$ the so-called \emph{Wigner distribution }of
$\psi_{h}$ defined on the cotangent bundle $T^{\ast}M$. Given a solution
$\psi\left(  t,\cdot\right)  =e^{ith\Delta/2}u\in L^{2}\left(  M\right)  $,
its Wigner distribution $W_{u}^{h}\left(  t,\cdot\right)  $ acts on smooth,
compactly supported test functions $a\in C_{c}^{\infty}\left(  T^{\ast
}M\right)  $ as:%
\begin{equation}
\left\langle W_{u}^{h}\left(  t,\cdot\right)  ,a\right\rangle
:=(\operatorname*{op}\nolimits_{h}\left(  a\right)  e^{ith\Delta
/2}u|e^{ith\Delta/2}u). \label{SQA}%
\end{equation}
Above, $\left(  \cdot|\cdot\right)  $ denotes the inner product in
$L^{2}\left(  M\right)  $ and $\operatorname*{op}\nolimits_{h}\left(
a\right)  $ stands for the semiclassical pseudodifferential operator of symbol
$a$ obtained by Weyl's quantization rule (see Section \ref{SecSSS} and the
references therein for precise definitions and further properties of these
objects). The Wigner distribution behaves, in some sense, as a joint position
and momentum density: it is real, although not necessarily positive, and its
marginals are precisely the position and momentum densities of $\psi$ (a
detailed presentation may be found, for instance, in the book \cite{Fo}).
Therefore, the limit of $\left\vert \psi_{h}\left(  t,\cdot\right)
\right\vert ^{2}$ may be recovered from that of $W_{u}^{h}\left(
t,\cdot\right)  $ simply by projecting on $M$.

There are different regimes in which the correspondence principle can be made
precise in the form of a rigorous result.\medskip

\noindent\textbf{The semiclassical limit.} Given a sequence of initial data
$\left(  u_{h}\right)  $ bounded in $L^{2}\left(  M\right)  $, consider the
corresponding Wigner distributions $W_{u_{h}}^{h}$ given by (\ref{SQA}). It is
by now well-known that the distributions $W_{u_{h}}^{h}\left(  t,\cdot\right)
$ converge as $h\rightarrow0^{+}$ (after possibly extracting a subsequence) to
a family of positive measures $\mu\left(  t,\cdot\right)  $, continuous in
time, usually called \emph{semiclassical }or\emph{ Wigner measures}. It turns
out that this limits are transported along the geodesic flow $\phi_{t}$ on
$T^{\ast}M$:\footnote{This is a classical result that has been revisited and
extended by many authors. A rigorous proof of this (and further results in
that direction) in the recent mathematical literature may be found, for
instance, in \cite{Ge91c, Li-Pau, Bu97a} (se also \cite{Robert}).}%
\begin{equation}
\lim_{h\rightarrow0^{+}}(\operatorname*{op}\nolimits_{h}\left(  a\right)
e^{ith\Delta/2}u_{h}|e^{ith\Delta/2}u_{h})=\int_{T^{\ast}M}a\left(  \phi
_{-t}\left(  x,\xi\right)  \right)  d\mu\left(  0,x,\xi\right)  .
\label{limSc}%
\end{equation}
Moreover, if $\left(  u_{h}\right)  $ oscillates at some characteristic
length-scale $h$ (see hypothesis (\ref{h-osc}) in Section \ref{SecR}) then the
position densities $|e^{ith\Delta/2}u_{h}|^{2}$ weakly converge towards the
marginal $\int_{T_{x}^{\ast}M}\mu\left(  t,x,d\xi\right)  $. This is the
precise sense in which we recover classical dynamics as $h\rightarrow0^{+}$ in
this particular setting.

This kind of result holds in any compact Riemannian manifold, regardless of
its particular geometric properties. However, only \emph{small times} (of
order $h$) are considered in the limit (\ref{limSc}). This prevents the
dispersive nature of the Schr\"{o}dinger flow to become effective. Since the
proof of (\ref{limSc}) relies essentially on Egorov's theorem,\ statement
(\ref{limSc}) still holds for times of order $T_{\text{E}}^{h}:=C\log\left(
1/h\right)  $;\footnote{This is the \emph{Ehrenfest time}, the time up to
which, in a general system, a wavepacket remains localized (see for instance,
\cite{BGP, BouzRobertEhrenfest, CombRobert, HJ99, HJ00, SchubertCMP} for
rigorous results in this direction).} that is, when rescaling the Wigner
distribution as $W_{u_{h}}^{h}\left(  T_{\text{E}}^{h}t,\cdot\right)
$.\medskip

\noindent\textbf{Eigenfunction limits.} Another approach, which gives results
that are valid for any time scale, consists of assuming that $M$ is compact
and taking as initial data eigenfunctions of $-\Delta$. If $\left(
\psi_{\lambda_{k}}\right)  $ is a sequence of normalized eigenfunctions,
$-\Delta\psi_{\lambda_{k}}=\lambda_{k}\psi_{\lambda_{k}}$ with $\lambda
_{k}\rightarrow\infty$, then the corresponding solutions to the
Schr\"{o}dinger equation (\ref{SE}) are $e^{ith\Delta/2}\psi_{\lambda_{k}%
}=e^{-ith\lambda_{k}/2}\psi_{\lambda_{k}}$. The associated Wigner
distributions act on $a\in C_{c}^{\infty}\left(  T^{\ast}M\right)  $ as:%
\begin{equation}
\left(  \operatorname*{op}\nolimits_{h}\left(  a\right)  \psi_{\lambda_{k}%
}|\psi_{\lambda_{k}}\right)  . \label{QL}%
\end{equation}
After setting $h=h_{k}=1/\sqrt{\lambda_{k}}$ and taking limits in (\ref{QL}),
a semiclassical measure is obtained (in this context, sometimes also called a
\emph{quantum limit}). Note that, since the Wigner distributions are
time-independent, the limits of (\ref{QL}) are \emph{uniform in time}.
Moreover, quantum limits are invariant under the geodesic flow and are
supported in the unit cosphere bundle $S^{\ast}M$. The main issue in this
setting is that of identifying the set of all possible invariant measures on
$S^{\ast}M$ that can be realized as a quantum limit.

This problem is, in general, very hard and depends heavily on the specific
geometry of the manifold under consideration. For instance, when the geodesic
flow is ergodic the celebrated Schnirelman theorem asserts that for a sequence
of eigenvalues of density one, (\ref{QL}) converges to the Liouville measure
on $S^{\ast}M$. Therefore, most sequences of eigenfunctions become
equidistributed on $M$ (see the original article of Schnirelman \cite{Shni}
and \cite{Zelditch87, CdV85, HMR, GeLei, RudSar, AnantAnosov, AnantNonn},
among many others, for various extensions and improvements). In the case of
completely integrable geodesic flow the situation is quite different. For
instance, when $\left(  M,g\right)  $ is the sphere $\mathbf{S}^{d}$ equipped
with the canonical metric Jakobson an Zelditch proved in \cite{JakZel} that
every invariant measure on $S^{\ast}\mathbf{S}^{d}$ may be realized as a
quantum limit for some sequence of normalized eigenfunctions.\footnote{An
extension of this result to general Compact Rank-One Symmetric Spaces can be
found in \cite{MaciaZoll}.} For a more comprehensive account of the results
quoted so far the reader may consult, for instance, \cite{DimSjos, Robert,
Zelditch96}.\medskip

\noindent\textbf{Intermediate time scales.} In this article we are interested
in an intermediate regime. We shall analyze the structure of semiclassical
measures arising as limits of Wigner distributions of solutions to a class of
Schr\"{o}dinger equations at time scales $t_{h}=\alpha_{h}$ tending to
infinity as $h\rightarrow0^{+}$; witch can be in principle much greater than
the Ehrenfest time. One should expect that the dispersive effects associated
to the Schr\"{o}dinger flow would have to be taken into
account.\footnote{Recent results, perhaps of a different flavor, on
semiclassical dynamics beyond the Ehrenfest time may be found, for instance,
in \cite{Fau07, SchubertCMP}.}

It turns out that the highly oscillating nature of the propagator
$e^{it\Delta/2}$ prevents in general that the rescaled Wigner distributions
$W_{u_{h}}^{h}\left(  \alpha_{h}t,\cdot\right)  $ converge for all
$t\in\mathbf{R}$. Therefore, we shall study the relations between
time-averages of (\ref{SQA}) and the semiclassical measures of their
corresponding sequences of initial data. The existence of these limits is
established in Theorem \ref{ThmE}; Theorem \ref{ThmInvEg} shows that they are
invariant by the geodesic flow (as was the case for eigenfunction limits) and
that a weak from of Egorov's theorem holds for time scales $\alpha
_{h}=o\left(  h^{-2}\right)  $.

Then, in order to get a more detailed description of these limits, we examine
some examples of manifolds with completely integrable geodesic flow: Zoll
manifolds, Euclidean space and the flat torus. Under some assumptions on the
initial data we prove that limits of time averages of (\ref{SQA}) are
expressed as averages under the geodesic flow of the semiclassical measure of
the initial states (see Theorem \ref{ThmZoll} and Propositions \ref{PropRd},
\ref{Prop AvTorus}, which are proved in Section \ref{SecAv}). This is again a
manifestation of the correspondence principle.

However, it turns out that such a behavior may fail in general, even in the
completely integrable case. In Proposition \ref{PropQuantTorus}, which is
proved in Section \ref{SecT}, we present an example of initial data for which
the semiclassical measure of the corresponding evolved states does not obey
the averaging rule mentioned above. Instead of this, they evolve following a
law related to the Schr\"{o}dinger flow, thus exhibiting a genuinely quantum
behavior. Moreover, we show that there is no longer a formula relating it to
the semiclassical measure $\mu_{0}$ of the sequence of initial states. It is
possible to construct sequences having the same $\mu_{0}$ but such that the
limits of the time averages of (\ref{SQA}) differ.

Finally, in Section \ref{SecV} we discuss how our results extend to more
general Schr\"{o}dinger equations.\medskip

\noindent\textbf{Notation and conventions. }In what follows, $\left(
M,g\right)  $ will always denote a connected, complete, $d$-dimensional smooth
Riemannian manifold possessing a semiclassical functional calculus (property
(F) in Section \ref{SecSSS}).

$T^{\ast}M$ and $S^{\ast}M$ stand for the cotangent and unit cosphere bundles
on $M$ respectively. Given a diffeomorphism $\Phi:M\rightarrow N$ between
smooth manifolds, we will denote by $\tilde{\Phi}:T^{\ast}M\rightarrow
T^{\ast}N$ the canonical diffeomorphisms induced by $\Phi$.

The Riemannian norm of a point $\left(  x,\xi\right)  \in T^{\ast}M$ is
denoted by $\left\Vert \xi\right\Vert _{x}$.

The geodesic flow on $T^{\ast}M$ is the Hamiltonian flow induced by the
Riemannian energy $\frac{1}{2}\left\Vert \xi\right\Vert _{x}^{2}$. It will be
denoted by $\phi_{t}$.

The Riemannian measure in $M$ will be denoted by $dm$. We shall write for
short $L^{2}\left(  M\right)  :=L^{2}\left(  M,dm\right)  $; the scalar
product of two functions $u,v\in L^{2}\left(  M\right)  $ will be written as
$\left(  u|v\right)  $.

The Riemannian gradient will be denoted by $\nabla$; the Laplacian is denoted
by $\Delta:=\operatorname*{div}\left(  \nabla\cdot\right)  $. It is a
self-adjoint operator on $L^{2}\left(  M\right)  $.

$\mathcal{M}\left(  T^{\ast}M\right)  $ (resp. $\mathcal{M}_{+}\left(
T^{\ast}M\right)  $) denotes the space of Radon measures (resp. the cone of
positive Radon measures) on $T^{\ast}M$. The space $\mathcal{M}\left(
T^{\ast}M\right)  $ may be identified, by Riesz's theorem, to the dual of the
space of continuous compactly supported functions $C_{c}\left(  T^{\ast
}M\right)  $.

A sequence of measures $\left(  \mu_{n}\right)  $ in $\mathcal{M}\left(
T^{\ast}M\right)  $ converges vaguely to some Radon measure $\mu$ as
$n\rightarrow\infty$ if and only if $\lim_{n\rightarrow\infty}\int_{T^{\ast}%
M}ad\mu_{n}=\int_{T^{\ast}M}ad\mu$ for every $n\rightarrow\infty$ and $a\in
C_{c}\left(  T^{\ast}M\right)  $.

A measure $\mu\in\mathcal{M}\left(  T^{\ast}M\right)  $ is invariant by a flow
$\phi_{t}$ on $T^{\ast}M$ if for any measurable set $X\subset T^{\ast}M$ one
has $\mu\left(  X\right)  =\mu\left(  \phi_{s}\left(  X\right)  \right)  $ for
every $s\in\mathbf{R}$. This can be equivalently stated as $\int_{T^{\ast}%
M}ad\mu=\int_{T^{\ast}M}a\circ\phi_{s}d\mu$ for every $s\in\mathbf{R}$ and
$a\in C_{c}\left(  T^{\ast}M\right)  $.

The space of compactly supported smooth functions on $T^{\ast}M$ will be
written as $C_{c}^{\infty}\left(  T^{\ast}M\right)  $; its dual, the space of
distributions on $T^{\ast}M$, will be denoted by $\mathcal{D}^{\prime}\left(
T^{\ast}M\right)  $. The duality bracket in $\mathcal{D}^{\prime}\left(
T^{\ast}M\right)  \times C_{c}^{\infty}\left(  T^{\ast}M\right)  $ will be
denoted by $\left\langle \cdot,\cdot\right\rangle $. Weak-$\ast$ convergence
in $\mathcal{D}^{\prime}\left(  T^{\ast}M\right)  $ will be simply referred to
as weak convergence.

Given a set $A\subset\mathbf{R}$, its characteristic function will be denoted
by $\mathbf{1}_{A}$.

\section{\label{SecR}Statement of the results}

Our first results describe some properties of the limits of Wigner
distributions at times $t=\alpha_{h}\rightarrow\infty$ corresponding to
solutions to (\ref{SE}) on a general Riemannian manifold. In Section
\ref{SecV} we comment on extensions of these results to more general
Schr\"{o}dinger equations.

We shall make some hypotheses on the initial states. As it is also the case
when dealing with the semiclassical limit, we shall assume that the admissible
sequences of initial data $\left(  u_{h}\right)  $ satisfy the $h$-oscillation
property:%
\begin{equation}
\limsup_{h\rightarrow0^{+}}\left\Vert \mathbf{1}_{\left(  -\infty,R\right)
}\left(  h^{2}\Delta\right)  \phi u_{h}\right\Vert _{L^{2}\left(  M\right)
}\rightarrow0,\qquad\text{as }R\rightarrow-\infty\text{, for every }\phi\in
C_{c}^{\infty}\left(  M\right)  . \label{h-osc}%
\end{equation}
When the spectrum of $\Delta$ is discrete, this roughly means that the energy
of $u_{h}$ is concentrated on Fourier modes corresponding to eigenvalues of
size at most $R/h^{2}$.

Moreover, we shall assume that their Wigner distributions converge to some
semiclassical measure $\mu_{0}\in\mathcal{M}_{+}\left(  T^{\ast}M\right)  $:
\begin{equation}
\lim_{h\rightarrow0^{+}}\left(  \operatorname*{op}\nolimits_{h}\left(
a\right)  u_{h}|u_{h}\right)  =\int_{T^{\ast}M}a\left(  x,\xi\right)  \mu
_{0}\left(  dx,d\xi\right)  ,\label{convID}%
\end{equation}
for every $a\in C_{c}^{\infty}\left(  T^{\ast}M\right)  $. This is always
achieved by some subsequence (provided that $\left(  u_{h}\right)  $ is
bounded in $L^{2}\left(  M\right)  $). See Proposition \ref{PropSCM} and, in
general, Section \ref{SecSSS} for notation and background concerning
pseudodifferential operators and Wigner distributions. 

Unless otherwise stated, we shall denote by $\left(  \alpha_{h}\right)  $ a
sequence of positive reals tending to infinity as $h\rightarrow0^{+}$.

\begin{theorem}
\label{ThmE}Let $\left(  u_{h}\right)  $ be a bounded sequence in
$L^{2}\left(  M\right)  $ satisfying hypotheses (\ref{h-osc}) and
(\ref{convID}). Then there exist a subsequence and a finite measure $\mu\in
L^{\infty}\left(  \mathbf{R}_{t}\mathbf{;}\mathcal{M}_{+}\left(  T^{\ast
}M\right)  \right)  $ such that the following statements hold.

i) For every $\varphi\in L^{1}\left(  \mathbf{R}\right)  $ and every $a\in
C_{c}^{\infty}\left(  T^{\ast}M\right)  $,
\begin{equation}
\lim_{h\rightarrow0^{+}}\int_{\mathbf{R}}\varphi\left(  t\right)
(\operatorname*{op}\nolimits_{h}\left(  a\right)  e^{i\alpha_{h}ht\Delta
/2}u_{h}|e^{i\alpha_{h}ht\Delta/2}u_{h})dt=\int_{\mathbf{R\times}T^{\ast}%
M}\varphi\left(  t\right)  a\left(  x,\xi\right)  \mu\left(  t,dx,d\xi\right)
dt. \label{conv}%
\end{equation}

ii) For every $\varphi\in L^{1}\left(  \mathbf{R}\right)  $ and $a\in
C_{c}\left(  M\right)  $ the evolved position densities satisfy:%
\begin{equation}
\lim_{h\rightarrow0+}\int_{\mathbf{R}\times M}\varphi\left(  t\right)
a\left(  x\right)  |e^{i\alpha_{h}ht\Delta/2}u_{h}\left(  x\right)
|^{2}dmdt=\int_{\mathbf{R\times}T^{\ast}M}\varphi\left(  t\right)  a\left(
x\right)  \mu\left(  t,dx,d\xi\right)  dt. \label{den}%
\end{equation}

\end{theorem}

In general, the convergence in (\ref{conv}) does not hold pointwise. Several
examples of such a behavior will be presented in our next results.

\begin{theorem}
\label{ThmInvEg}Let $\mu$ and $\mu_{0}$ be obtained as a limit (\ref{conv})
and (\ref{convID}), respectively. Then the following hold.\medskip

i) For almost every $t\in\mathbf{R}$, the measure $\mu\left(  t,\cdot\right)
$ is invariant under the geodesic flow $\phi_{s}$, \emph{i.e.}%
\begin{equation}
\mu\left(  t,\phi_{s}\left(  \Omega\right)  \right)  =\mu\left(
t,\Omega\right)  ,\qquad\text{for every }s\in\mathbf{R}\text{ and }%
\Omega\subset T^{\ast}M\text{ measurable.} \label{invariance}%
\end{equation}

ii) If $a\in C_{c}^{\infty}\left(  T^{\ast}M\right)  $ is invariant under the
classical flow and $\alpha_{h}=o\left(  1/h^{2}\right)  $ then the following
holds pointwise, \emph{for every }$t\in\mathbf{R}$:%
\begin{equation}
\lim_{h\rightarrow0^{+}}(\operatorname*{op}\nolimits_{h}\left(  a\right)
e^{i\alpha_{h}ht\Delta/2}u_{h}|e^{i\alpha_{h}ht\Delta/2}u_{h})=\int_{T^{\ast
}M}a\left(  x,\xi\right)  \mu_{0}\left(  dx,d\xi\right)  . \label{CI}%
\end{equation}

\end{theorem}

\begin{remark}
The restriction $\alpha_{h}=o\left(  1/h^{2}\right)  $ in part ii) of the
theorem can be removed in some cases, as the Euclidean space $\mathbf{R}^{d}$
or the flat torus $\mathbf{T}^{d}$. Its presence is related to the commutation
properties with $\Delta$ of $\operatorname*{op}\nolimits_{h}\left(  a\right)
$ when $a$ is invariant. Further details are given in Remark \ref{RmkComm}, in
Section \ref{SecZ}.
\end{remark}

Part i) is a consequence of the time averaging over large time intervals. Part
ii) establishes that we can still keep track of the pointwise behavior of the
Wigner distributions at large time scales, provided we test them against an
invariant classical symbol. This can be interpreted as a weak form of Egorov's
theorem for long times. An analogue of Theorem \ref{ThmE} and of part i) of
Theorem \ref{ThmInvEg} in terms of microlocal defect measures (see for
instance \cite{BurqMDAsterisque, GerardMDM} for background) can be found in
\cite{DeGeLe}.

In order to obtain a more detailed description and, in particular, to derive
formulas that allow to compute $\mu$ in terms of the semiclassical measure of
the initial data $\mu_{0}$, we must restrict the geometry of the manifolds
under consideration. We shall consider examples of manifolds with completely
integrable geodesic flow

We first consider the case of \emph{Zoll manifolds} (that is, manifolds all
whose geodesics are closed). We refer to the book \cite{Besse} for a
comprehensive study of this geometric hypothesis. Such manifolds are compact,
and the restriction of the geodesic flow $\phi_{t}$ to the unit cosphere
bundle $S^{\ast}M$ is periodic. Given a function $a\in C_{c}\left(  T^{\ast
}M\right)  $ we write $\left\langle a\right\rangle $ to denote the average of
$a$ along the geodesic flow:%
\begin{equation}
\left\langle a\right\rangle \left(  x,\xi\right)  :=\lim_{T\rightarrow\infty
}\frac{1}{T}\int_{0}^{T}a\left(  \phi_{s}\left(  x,\xi\right)  \right)  ds.
\label{aver}%
\end{equation}
Since, by the homogeneity of the flow, every trajectory is periodic, the limit
above always exists. Moreover, $\left\langle a\right\rangle $ is bounded and measurable.

\begin{theorem}
\label{ThmZoll}Suppose $\left(  M,g\right)  $ is a manifold all of whose
geodesics are closed and $\alpha_{h}=o\left(  1/h^{2}\right)  $. Let $\mu_{0}$
be the semiclassical measure given by (\ref{convID}) for some sequence of
initial data satisfying (\ref{h-osc}). If $\mu_{0}\left(  \left\{
\xi=0\right\}  \right)  =0$ then any limit $\mu$ given by (\ref{conv}) is
characterized by:%
\begin{equation}
\int_{T^{\ast}M}a\left(  x,\xi\right)  \mu\left(  t,dx,d\xi\right)
=\int_{T^{\ast}M}\left\langle a\right\rangle \left(  x,\xi\right)  \mu
_{0}\left(  dx,d\xi\right)  ,\qquad\text{for a.e. }t\in\mathbf{R}\text{.}
\label{averZ}%
\end{equation}

\end{theorem}

In particular, if $\left(  u_{h}\right)  $ is such that $\mu_{0}\left(
x,\xi\right)  =\delta_{x_{0}}\left(  x\right)  \delta_{\xi_{0}}\left(
\xi\right)  $ for some $\left(  x_{0},\xi_{0}\right)  \in T^{\ast}%
M\setminus\left\{  0\right\}  $ --$\left(  u_{h}\right)  $ is then called a
wave-packet, see Proposition \ref{PropCS}-- then $\mu=\delta_{\gamma}$ is the
Dirac delta on the geodesic $\gamma$ issued from $\left(  x_{0},\xi
_{0}\right)  $.\footnote{That is, $\delta_{\gamma}$ is the unique invariant
probability measure on $T^{\ast}M$ which is concentrated on $\gamma$. This is
sometimes also called the orbit measure corresponding to $\gamma$.} From this,
using a diagonal argument, it is clear that if $\left(  M,g\right)  $ is a
Zoll manifold then every measure on $T^{\ast}M\setminus\left\{  0\right\}  $
that is invariant under the geodesic flow can be realized as a limit
(\ref{conv}) for some sequence of initial data. This can be seen as a time
dependent version of the result of Jakobson and Zelditch \cite{JakZel} for
eigenfunctions of the Laplacian on the sphere we quoted in the introduction.
Actually, Theorem \ref{ThmZoll} can also be applied to obtain results on
quantum limits; in particular, it can be used to extend the result in
\cite{JakZel} to a general Compact Rank-One Symmetric Space (see
\cite{MaciaZoll}).

\begin{remark}
As it will be clear from the proof, Theorem \ref{ThmZoll} holds locally in the
following sense. If the geometric hypothesis on $\left(  M,g\right)  $ is
replaced by the weaker: \emph{there exist an open set }$X\subset T^{\ast}%
M$\emph{, invariant under the geodesic flow, such that }$\phi_{s}|_{X}$\emph{
is periodic on each of the cospheres }$\left\Vert \xi\right\Vert _{x}%
=$\emph{constant}, then (\ref{averZ}) holds for every $a\in C_{c}^{\infty
}\left(  X\right)  $.
\end{remark}

The proof of the theorem follows from a general result which relates the
smoothness properties of the averages $\left\langle a\right\rangle $ to the
time-pointwise behavior of Wigner distributions (\emph{cf. }Lemma
\ref{Lemma av} in Section \ref{SecZ}, which is of independent interest).

The consequence of the corresponding result on Euclidean space is trivial.

\begin{proposition}
\label{PropRd}Suppose $\left(  M,g\right)  =\left(  \mathbf{R}^{d}%
,\text{\emph{can}}\right)  $ and $\left(  u_{h}\right)  $ is a sequence that
satisfies (\ref{h-osc}) and (\ref{convID}). If its semiclassical measure
$\mu_{0}$ satisfies $\mu_{0}\left(  \left\{  \xi=0\right\}  \right)  =0$ then
any measure $\mu$ given by (\ref{conv}) vanishes identically. In other words,%
\[
\lim_{h\rightarrow0^{+}}\int_{\mathbf{R}}\varphi\left(  t\right)
(\operatorname*{op}\nolimits_{h}\left(  a\right)  e^{i\alpha_{h}ht\Delta
/2}u_{h}|e^{i\alpha_{h}ht\Delta/2}u_{h})dt=0,\qquad\text{for every }a\in
C_{c}^{\infty}\left(  T^{\ast}\mathbf{R}^{d}\right)  \text{.}%
\]

\end{proposition}

Note that Proposition \ref{PropRd} can also be deduced from the $H^{1/2}%
$-regularizing effect of the Schr\"{o}dinger equation (see for instance
\cite{ConstSaut}).

\begin{remark}
The condition $\mu_{0}\left(  \left\{  \xi=0\right\}  \right)  =0$ roughly
means that the sequence $\left(  u_{h}\right)  $ cannot develop oscillations
at frequencies lower than $h^{-1}$. It holds when%
\begin{equation}
\limsup_{h\rightarrow0^{+}}\left\Vert \mathbf{1}_{\left(  \delta,0\right]
}\left(  h^{2}\Delta\right)  \phi u_{h}\right\Vert _{L^{2}\left(  M\right)
}\rightarrow0,\qquad\text{as }\delta\rightarrow0^{-}\text{, for every }\phi\in
C_{c}^{\infty}\left(  M\right)  \text{.} \label{stricho}%
\end{equation}

\end{remark}

\begin{remark}
On any Riemannian manifold, one easily checks that the limit (\ref{conv})
corresponding to the constant sequence $u_{h}:=f\in L^{2}\left(  M\right)  $
is given, for every $t\in\mathbf{R}$, by:%
\[
\mu\left(  t,x,\xi\right)  =|e^{it\Delta/2}f\left(  x\right)  |^{2}%
dx\delta_{0}\left(  \xi\right)  .
\]
Thus, the conclusions of Theorem \ref{ThmZoll} and Proposition \ref{PropRd}
may not hold when $\mu_{0}\left(  \left\{  \xi=0\right\}  \right)  \not =0$.
\end{remark}

\begin{remark}
Analogues of these results hold for Schr\"{o}dinger equations with a
potential, see Theorem \ref{ThmPeriodic} and Remark \ref{RmkV} in Section
\ref{SecV}.
\end{remark}

Our last set of results deal with the flat torus $\left(  \mathbf{T}%
^{d},\text{can}\right)  $. We shall identify $\mathbf{T}^{d}$ with the
quotient $\mathbf{R}^{d}/\left(  2\pi\mathbf{Z}\right)  ^{d}$ and $T^{\ast
}\mathbf{T}^{d}$ to $\mathbf{T}^{d}\times\mathbf{R}^{d}$. Consider the set of
\emph{resonant frequencies}:
\[
\Omega:=\left\{  \xi\in\mathbf{R}^{d}\text{ }:\text{ }k\cdot\xi=0\text{ for
some }k\in\mathbf{Z}^{d}\setminus\left\{  0\right\}  \right\}  .
\]
We again get an averaging type result, provided our sequence of initial data
does not concentrate on $\Omega$.

\begin{proposition}
\label{Prop AvTorus}Suppose $\mu$ and $\mu_{0}$ are given respectively by
(\ref{conv}) and (\ref{convID}) for some sequence $\left(  u_{h}\right)  $
bounded in $L^{2}\left(  \mathbf{T}^{d}\right)  $ satisfying (\ref{h-osc}). If
$\mu_{0}\left(  \mathbf{T}^{d}\times\Omega\right)  =0$ then, for a.e.
$t\in\mathbf{R}$ and every $a\in C_{c}^{\infty}\left(  T^{\ast}\mathbf{T}%
^{d}\right)  $,
\[
\int_{\mathbf{T}^{d}}a\left(  x,\xi\right)  \mu\left(  t,dx,d\xi\right)
=\int_{\mathbf{T}^{d}}\left\langle a\right\rangle \left(  x,\xi\right)
\mu_{0}\left(  dx,d\xi\right)  .
\]

\end{proposition}

As we mentioned in the introduction, the results obtained so far reflect that
the correspondence principle holds. It turns out that this is no longer the
case if $\mu_{0}$ charges $\mathbf{T}^{d}\times\Omega$.

\begin{proposition}
\label{PropQuantTorus}Let $\xi_{0}\in\mathbf{Z}^{d}$, $\varrho\in L^{2}\left(
\mathbf{T}^{d}\right)  $ and $\alpha_{h}=1/h$. Then there exist sequences
$\left(  u_{h}\right)  $ and $\left(  v_{h}\right)  $ whose semiclassical
measure is:%
\begin{equation}
\mu_{0}\left(  x,\xi\right)  =\left\vert \varrho\left(  x\right)  \right\vert
^{2}dx\delta_{\xi_{0}}\left(  \xi\right)  , \label{mu0osc}%
\end{equation}
but such that the limiting semiclassical measure (in the sense of
(\ref{conv})) for $\left(  e^{it\Delta/2}u_{h}\right)  $ is:%
\begin{equation}
\mu_{\left(  u_{h}\right)  }\left(  t,x,\xi\right)  =\left\langle
|e^{it\Delta/2}\varrho|^{2}\right\rangle \left(  x\right)  dx\delta_{\xi^{0}%
}\left(  \xi\right)  , \label{mu1}%
\end{equation}
(with $\left\langle \cdot\right\rangle $ defined by (\ref{aver})), whereas
that of $\left(  e^{it\Delta/2}v_{h}\right)  $ is given by:
\begin{equation}
\mu_{\left(  v_{h}\right)  }\left(  t,x,\xi\right)  =\frac{1}{\left(
2\pi\right)  ^{d}}\left(  \int_{\mathbf{T}^{d}}\left\vert \varrho\left(
y\right)  \right\vert ^{2}dy\right)  dx\delta_{\xi^{0}}\left(  \xi\right)  .
\label{mu2}%
\end{equation}

\end{proposition}

We can extract two consequences of this result. First, that the measures
$\mu\left(  t,\cdot\right)  $ may have an explicit dependence on $t$, which is
related to the Schr\"{o}dinger flow and does not depend exclusively on the
classical dynamics. Second, that no formula exists in general relating
$\mu_{0}$ and $\mu$ in the case that $\mu_{0}$ charges the resonant set
$\mathbf{T}^{d}\times\Omega$. In fact, $\mu$ depends on the way in which
concentration of the sequence of initial data takes place on $\mathbf{T}%
^{d}\times\Omega$. A more detailed study requires the introduction of
two-microlocal objects describing such a concentration and will be presented
in \cite{MaciaTorus}.

\section{\label{SecSSS}Semiclassical measures}

In this section we shall recall the necessary notions of semiclassical
pseudodifferential calculus and semiclassical measures that will be needed in
the sequel. We shall closely follow the presentation in \cite{GeLei}. Unless
otherwise specified, we implicitly refer to \cite{GeLei} for complete proofs
of the results presented in this section.

The classical \emph{Weyl quantization rule} on $\mathbf{R}^{d}$ associates to
any function $a\in C_{c}^{\infty}\left(  \mathbf{R}^{d}\times\mathbf{R}%
^{d}\right)  $ and any $h>0$ an operator $\operatorname*{op}\nolimits_{h}%
\left(  a\right)  $ acting on $u\in C_{c}^{\infty}\left(  \mathbf{R}%
^{d}\right)  $ as:%
\[
\operatorname*{op}\nolimits_{h}\left(  a\right)  u\left(  x\right)
:=\int_{\mathbf{R}^{d}}\int_{\mathbf{R}^{d}}a\left(  \frac{x+y}{2}%
,h\xi\right)  u\left(  y\right)  e^{i\left(  x-y\right)  \cdot\xi}dy\frac
{d\xi}{\left(  2\pi\right)  ^{d}}.
\]
It turns out that, under suitable growth conditions on $a$, the operators
$\operatorname*{op}\nolimits_{h}\left(  a\right)  $ are uniformly bounded in
$L^{2}\left(  \mathbf{R}^{d}\right)  $ when $h$ ranges any compact set of the
positive reals.

In order to extend this rule to functions $a\in C^{\infty}\left(  T^{\ast
}M\right)  $ we shall do the following. Let $\kappa:U\subset\mathbf{R}%
^{d}\rightarrow V\subset M$ be a coordinate patch; assume that $a$ is
supported on $T^{\ast}M|_{V}$. Then define, for every $h>0$, an operator
$\operatorname*{op}\nolimits_{h}\left(  a\right)  $ by the formula:%
\[
\left(  \operatorname*{op}\nolimits_{h}\left(  a\right)  u\right)  \circ
\kappa:=\theta\operatorname*{op}\nolimits_{h}\left(  a\circ\tilde{\kappa
}\right)  \left(  \theta u\circ\kappa\right)  ,
\]
where $a\circ\tilde{\kappa}$ is the expression of $a$ in the coordinates
$\kappa$ and $\theta\in C_{c}^{\infty}\left(  V\right)  $ is identically equal
to one on the projection of $\operatorname*{supp}a$ on $M$. To deal with the
general case it suffices to decompose the function $a$ in compactly supported
components using a partition of unity.

In what follows, we shall assume that $a\in C_{c}^{\infty}\left(  T^{\ast
}M\right)  $. The operators $\operatorname*{op}\nolimits_{h}\left(  a\right)
$ are called \emph{semiclassical pseudodifferential operators of symbol }$a$.
The following facts are well known.

\begin{itemize}
\item[(A)] The operators $\operatorname*{op}\nolimits_{h}\left(  a\right)  $
are bounded in $L^{2}\left(  M\right)  $ with norm:%
\begin{equation}
\left\Vert \operatorname*{op}\nolimits_{h}\left(  a\right)  \right\Vert
_{\mathcal{L}\left(  L^{2}\left(  M\right)  \right)  }\leq C\left\Vert
a\right\Vert _{C^{d+1}\left(  T^{\ast}M\right)  }, \label{opaBound}%
\end{equation}
the constant $C>0$ being uniform in $h\in\left(  0,1\right]  $.

\item[(B)] The family $\operatorname*{op}\nolimits_{h}\left(  a\right)  $ of
operators is not completely determined by the function $a$ -- in fact, the
result may depend on the partition of unity, the coordinate patches and the
cut-off functions $\theta$ chosen. However the $L^{2}\left(  M\right)
$-operator norm of the difference of any two families defined from $a$ by
means of the above procedure tends to zero as $h\rightarrow0^{+}$.

\item[(C)] The Laplacian of $\left(  M,g\right)  $ may be expressed in terms
of semiclassical pseudodifferential operators. One easily checks that
\[
-h^{2}\Delta=\operatorname*{op}\nolimits_{h}\left(  p\right)
+ih\operatorname*{op}\nolimits_{h}\left(  r\right)  +h^{2}\operatorname*{op}%
\nolimits_{h}\left(  m\right)  ,
\]
where $m\in C^{\infty}\left(  M\right)  $ is a function of $x$ alone depending
only on the derivatives up to order two of the Riemannian metric $g$. In a
coordinate chart $\kappa$, the functions $p$, $r$ are given by:%
\[
\left(  p\circ\tilde{\kappa}\right)  \left(  x,\xi\right)  :=\sum_{i,j=1}%
^{d}g^{ij}\left(  x\right)  \xi^{i}\xi^{j},\quad\left(  r\circ\tilde{\kappa
}\right)  \left(  x,\xi\right)  :=\frac{1}{\rho\left(  x\right)  }\sum
_{i,j=1}^{d}g^{ij}\left(  x\right)  \partial_{x_{i}}\rho\left(  x\right)
\xi_{j},
\]
where $\rho:=\sqrt{\det g}$. Therefore, $p$ coincides with the squared
Riemannian norm $\left\Vert \xi\right\Vert _{x}^{2}$ and
\begin{equation}
r=\frac{1}{2}\left\{  p,\log\rho\right\}  , \label{defr}%
\end{equation}
where $\left\{  \cdot,\cdot\right\}  $ stands for the Poisson bracket induced
by the canonical symplectic structure on $T^{\ast}M$.
\end{itemize}

The Weyl quantization rule enjoys a powerful symbolic calculus (see
\cite{DimSjos, MartinezBook} for a thorough description). Some particular
cases are the following.

\begin{itemize}
\item[(D)] \emph{Commutators.} For every $a\in C_{c}^{\infty}\left(  T^{\ast
}M\right)  $ and $h>0$ there exists an operator $s_{h}\in\mathcal{L}\left(
L^{2}\left(  M\right)  \right)  $ such that:%
\begin{equation}
\left[  \operatorname*{op}\nolimits_{h}\left(  a\right)  ,-h^{2}\Delta\right]
=\frac{h}{i}\operatorname*{op}\nolimits_{h}\left(  \left\{  a,p\right\}
\right)  +h^{2}\operatorname*{op}\nolimits_{h}\left(  \left\{  a,r\right\}
\right)  +s_{h}, \label{CS}%
\end{equation}
and $\left\Vert s_{h}\right\Vert _{\mathcal{L}\left(  L^{2}\left(  M\right)
\right)  }\leq Ch^{3}$.

\item[(E)] \emph{Adjoints.} If $a\in C_{c}^{\infty}\left(  T^{\ast}M\right)  $
is real then $\operatorname*{op}_{h}\left(  a\right)  $ is self-adjoint in
$L^{2}\left(  M\right)  $.
\end{itemize}

Finally, we shall assume that our manifold $\left(  M,g\right)  $ possesses a
semiclassical functional calculus. More precisely, that the following holds:

\begin{itemize}
\item[(F)] \emph{Functional calculus.} For every $\sigma\in C_{c}^{\infty
}\left(  \mathbf{R}\right)  $ the following holds:
\begin{equation}
\sigma\left(  -h^{2}\Delta\right)  =\operatorname*{op}\nolimits_{h}\left(
\sigma\circ p\right)  +z_{h}, \label{CF}%
\end{equation}
with $\left\Vert z_{h}\right\Vert _{\mathcal{L}\left(  L^{2}\left(  M\right)
\right)  }\leq Ch$.
\end{itemize}

This is known to hold when $M$ is compact (see \cite{BGT04}), and has been
proved for Euclidean spaces in \cite{Robert} and recently for manifolds with
ends in \cite{BoucFC}.

Given a function $u\in L^{2}\left(  M\right)  $ we define its \emph{Wigner
distribution} $w_{u}^{h}\in\mathcal{D}^{\prime}\left(  T^{\ast}M\right)  $
acting on test functions $a\in C_{c}^{\infty}\left(  T^{\ast}M\right)  $ as:%
\[
\left\langle w_{u}^{h},a\right\rangle :=\left(  \operatorname*{op}%
\nolimits_{h}\left(  a\right)  u|u\right)  .
\]
Property (E) of the Weyl quantization ensures that $w_{u}^{h}$ is real.
Moreover, the following result holds.

\begin{proposition}
\label{PropSCM}Let $\left(  u_{h}\right)  $ be a bounded sequence in
$L^{2}\left(  M\right)  $. Then for some subsequence (which we do not relabel)
the Wigner distributions $w_{u_{h}}^{h}$ converge to a finite, positive Radon
measure $\mu\in\mathcal{M}_{+}\left(  T^{\ast}M\right)  $:%
\begin{equation}
\lim_{h\rightarrow0^{+}}\left\langle w_{u_{h}}^{h},a\right\rangle
=\int_{T^{\ast}M}a\left(  x,\xi\right)  \mu_{0}\left(  dx,d\xi\right)
,\qquad\text{for all }a\in C_{c}^{\infty}\left(  T^{\ast}M\right)  \text{.}
\label{Limit SCM}%
\end{equation}

\end{proposition}

Note that property (B) of $\operatorname*{op}\nolimits_{h}\left(
\cdot\right)  $ ensures that the limit $\mu_{0}$ does not depend on the
partitions of unity, coordinate charts, and cut-off functions used to define
$\operatorname*{op}\nolimits_{h}\left(  a\right)  $.

Whenever (\ref{Limit SCM}) holds, we say that $\mu_{0}$ is the
\emph{semiclassical measure }of the sequence $\left(  u_{h}\right)  $. If in
addition the sequence satisfies the $h$-oscillation property (\ref{h-osc})
then $\left\vert u_{h}\right\vert ^{2}dm$ tends to the projection on $M$ of
$\mu_{0}$ as $h\rightarrow0^{+}$.

\begin{proposition}
\label{PropTight}Let $\mu_{0}$ be the semiclassical measure of an
$h$-oscillating sequence $\left(  u_{h}\right)  $. Suppose that
\[
\left\vert u_{h}\right\vert ^{2}dm\rightharpoonup\nu\qquad\text{vaguely in
}\mathcal{M}_{+}\left(  M\right)  \text{ as }h\rightarrow0^{+}\text{.}%
\]
Then
\[
\int_{T_{x}^{\ast}M}\mu_{0}\left(  x,d\xi\right)  =\nu\left(  x\right)  .
\]

\end{proposition}

\begin{proof}
The proof of this result combines that of Proposition 1.6 in \cite{GeLei} with
the functional calculus formula (\ref{CF}). Working in coordinates
$\kappa:U\subset\mathbf{R}^{d}\rightarrow V\subset M$ and following exactly
the reasoning in \cite{GeLei}, Proposition 1.6, we deduce that the conclusion
holds provided%
\[
\limsup_{h\rightarrow0^{+}}\int_{\left\vert \xi\right\vert >R/h}\left\vert
\widehat{\theta u_{h}\circ\kappa}\left(  \xi\right)  \right\vert ^{2}%
d\xi\rightarrow0,\text{ as }R\rightarrow\infty,
\]
for any $\theta\in C_{c}^{\infty}\left(  V\right)  $. Using the functional
calculus formula (\ref{CF}) we deduce that this condition is satisfied
whenever (\ref{h-osc}) holds.
\end{proof}

We conclude this review of semiclassical measures examining a specific
computation of the semiclassical measure of a wave-packet. Let $\left(
x_{0},\xi_{0}\right)  \in T^{\ast}M$ and $\left(  U,\kappa\right)  $ a
coordinate system centered at $x_{0}$ (\emph{i.e. }$0\in U$ and $\kappa\left(
0\right)  =x_{0}$). Let $\varrho\in C_{c}^{\infty}\left(  \mathbf{R}%
^{d}\right)  $ be supported in $U$ and identically equal to one near the
origin and let $\varphi\in C^{\infty}\left(  M\right)  $ be a function such
that for $x\in U$,%
\[
\varphi\left(  \kappa\left(  x\right)  \right)  =\kappa^{\ast}\left(  \xi
_{0}\right)  \cdot x+i\left\vert x\right\vert ^{2},
\]
where $\kappa^{\ast}\left(  \xi_{0}\right)  $ stands for the pull-back by
$\kappa$ of the covector $\xi_{0}\in T_{x_{0}}^{\ast}M$. Define $\rho_{h}\in
C_{c}^{\infty}\left(  \kappa\left(  U\right)  \right)  $ as $\rho_{h}\left(
\kappa\left(  x\right)  \right)  :=\varrho\left(  x/h^{1/2}\right)  $ and
$v_{h}\in L^{2}\left(  M\right)  $ as%
\[
v_{h}\left(  x\right)  :=Ch^{-d/4}\rho_{h}\left(  x\right)  e^{i\varphi\left(
x\right)  /h},
\]
where $C>0$ is chosen to have $\left\Vert v_{h}\right\Vert _{L^{2}\left(
M\right)  }=1$.

The sequence $\left(  v_{h}\right)  $ is called a \emph{wave-packet (or a
coherent state) centered at }$\left(  x_{0},\xi_{0}\right)  $. A simple
computation shows the following.

\begin{proposition}
\label{PropCS}The sequence $\left(  v_{h}\right)  $ is $h$-oscillatory and has
a semiclassical measure $\mu_{0}=\delta_{\left(  x_{0},\xi_{0}\right)  }$.
\end{proposition}

Using an orthogonality property of semiclassical measures (see \cite{Ge91c},
Proposition 3.3) and the preceding result one sees that every linear
combination of delta measures in $T^{\ast}M$ can be realized as the
semiclassical measure of some sequence in $L^{2}\left(  M\right)  $. Since
these combinations of point masses are dense in $\mathcal{M}_{+}\left(
T^{\ast}M\right)  $, by the Krein-Milman theorem, we conclude that \emph{every
finite, positive Radon measure on }$T^{\ast}M$ \emph{can be realized as the
semiclassical measure for some sequence in }$L^{2}\left(  M\right)  $.

\section{\label{SecZ}Proof of Theorems \ref{ThmE} and \ref{ThmInvEg}}

\begin{proof}
[\textbf{Proof of Theorem \ref{ThmE}}]Let $\psi_{h}\left(  t,x\right)
:=e^{i\alpha_{h}ht\Delta/2}u_{h}\left(  x\right)  $ and consider the
corresponding sequence of time-space Wigner distributions $W_{h}\in
\mathcal{D}^{\prime}\left(  T^{\ast}\left(  \mathbf{R}\times M\right)
\right)  $ defined by%
\[
\left\langle W_{h},b\right\rangle :=\left(  \operatorname*{op}\nolimits_{h}%
\left(  b_{h}\right)  \psi_{h}|\psi_{h}\right)  _{L^{2}\left(  \mathbf{R}%
\times M\right)  },
\]
where, for $b\in C_{c}^{\infty}\left(  T^{\ast}\left(  \mathbf{R}\times
M\right)  \right)  $ we have written $b_{h}\left(  t,x,\tau,\xi\right)
:=b\left(  t,x,\tau/\alpha_{h},\xi\right)  $. It is easy to check that
sequence $\left(  W_{h}\right)  $ is bounded in $\mathcal{D}^{\prime}\left(
T^{\ast}\left(  \mathbf{R}\times M\right)  \right)  $, therefore it is
possible to extract a subsequence (which we shall not relabel) such that%
\[
\lim_{h\rightarrow0^{+}}\left\langle W_{h},b\right\rangle =\int_{T^{\ast
}\left(  \mathbf{R}\times M\right)  }b\left(  t,x,\tau,\xi\right)  d\tilde
{\mu}\left(  t,x,\tau,\xi\right)  .
\]
It turns out (see \cite{Bu97a, GeLei}) that the limit $\tilde{\mu}$ is a
positive Radon measure on $T^{\ast}\left(  \mathbf{R}\times M\right)  $. Let
$\varphi,\chi\in C_{c}^{\infty}\left(  \mathbf{R}\right)  $, with $0\leq
\chi\leq1$ and $\chi|_{\left(  -1,1\right)  }\equiv1$. For every $a\in
C_{c}^{\infty}\left(  T^{\ast}M\right)  $ we can write:%
\begin{equation}
\int_{\mathbf{R}}\varphi\left(  t\right)  \left(  \operatorname*{op}%
\nolimits_{h}\left(  a\right)  \psi_{h}|\psi_{h}\right)  dt=\left(
\operatorname*{op}\nolimits_{h}\left(  b_{h}^{R}\right)  \psi_{h}|\psi
_{h}\right)  _{L^{2}\left(  \mathbf{R}\times M\right)  }+r\left(  R,h\right)
, \label{Fl}%
\end{equation}
where $b_{h}^{R}\left(  t,x,\tau,\xi\right)  :=\varphi\left(  t\right)
\chi\left(  \tau/\alpha_{h}R\right)  a\left(  x,\xi\right)  $ and the
remainder $r$ is defined as follows. Set $\sigma_{R}\left(  \tau\right)
:=\sqrt{1-\chi\left(  \tau/R\right)  }$; standard arguments of semiclassical
pseudodifferential calculus give%
\[
r\left(  R,h\right)  =\int_{\mathbf{R}}\varphi\left(  t\right)
(\operatorname*{op}\nolimits_{h}\left(  a\right)  \sigma_{R}\left(  \frac
{h}{\alpha_{h}}D_{t}\right)  \psi_{h}|\sigma_{R}\left(  \frac{h}{\alpha_{h}%
}D_{t}\right)  \psi_{h})_{L^{2}\left(  \mathbf{R}\times M\right)
}dt+\mathcal{O}\left(  h^{2}\right)  .
\]
We have used the notation $\sigma_{R}\left(  \frac{h}{\alpha_{h}}D_{t}\right)
$ to denote the operator $\operatorname*{op}_{\frac{h}{\alpha_{h}}}\left(
\sigma_{R}\right)  $ acting on functions defined on $\mathbf{R}_{t}$. Clearly,
$\sigma_{R}\left(  \frac{h}{\alpha_{h}}D_{t}\right)  \psi_{h}=\sigma
_{R}\left(  h^{2}\Delta/2\right)  \psi_{h}$; therefore,%
\[
\left\vert r\left(  R,h\right)  \right\vert \leq C_{a,\varphi}\left\Vert
\sigma_{R}\left(  h^{2}\Delta/2\right)  \psi_{h}\right\Vert _{L^{2}\left(
M\right)  }^{2}+\mathcal{O}\left(  h^{2}\right)  ,
\]
and (\ref{h-osc}) ensures that $\limsup_{h\rightarrow0^{+}}r\left(
R,h\right)  $ tends to $0$ as $R\rightarrow\infty$. Taking limits in
(\ref{Fl}), first in $h\rightarrow0^{+}$ then $R\rightarrow\infty$, we
conclude:%
\begin{equation}
\lim_{h\rightarrow0^{+}}\int_{\mathbf{R}}\varphi\left(  t\right)  \left(
\operatorname*{op}\nolimits_{h}\left(  a\right)  \psi_{h}|\psi_{h}\right)
dt=\int_{T^{\ast}\left(  \mathbf{R}\times M\right)  }\varphi\left(  t\right)
a\left(  x,\xi\right)  d\tilde{\mu}\left(  t,x,\tau,\xi\right)  .
\label{limtilda}%
\end{equation}
Note that, because of the bound:%
\[
\sup_{t\in\mathbf{R}}\left\vert (\operatorname*{op}\nolimits_{h}\left(
a\right)  e^{i\alpha_{h}ht\Delta/2}u_{h}|e^{i\alpha_{h}ht\Delta/2}%
u_{h})\right\vert \leq C\left\Vert a\right\Vert _{C^{d+1}\left(  T^{\ast
}M\right)  }\left\Vert u_{h}\right\Vert _{L^{2}\left(  M\right)  }^{2},
\]
convergence in (\ref{limtilda}) actually takes place for any $\varphi\in
L^{1}\left(  \mathbf{R}\right)  $ and the limit is in $L^{\infty}\left(
\mathbf{R};\mathcal{M}_{+}\left(  T^{\ast}M\right)  \right)  $. Therefore, the
measure $\mu\left(  t,x,\xi\right)  :=\int_{\mathbf{R}}\tilde{\mu}\left(
t,x,d\tau,\xi\right)  $ fulfills the requirements of i).

We now prove ii). First remark that we cannot directly derive (\ref{den}) from
part i), since test functions depending only on $x$ are not compactly
supported in $T^{\ast}M$. We start noticing that $\left\vert e^{i\alpha
_{h}ht\Delta/2}u_{h}\right\vert ^{2}$ is bounded in $L^{\infty}\left(
\mathbf{R};L^{1}\left(  M\right)  \right)  $; this ensures existence of the
limit in (\ref{den}), eventually for a subsequence. In order to identify the
limit it is better to work locally in a coordinate patch $\kappa
:U\subset\mathbf{R}^{d}\rightarrow V\subset M$. From the functional calculus
identity (\ref{CF}) and the $h$-oscillation hypothesis (\ref{h-osc}) one
deduces that, for any $\theta\in C_{c}^{\infty}\left(  V\right)  $ and
$\varphi\in L^{1}\left(  \mathbf{R}\right)  $, the sequence $\left(
\theta\psi_{h}\circ\kappa\right)  $ enjoys the (euclidean) $h$-oscillation
property:%
\begin{equation}
\limsup_{h\rightarrow0^{+}}\int_{\mathbf{R}}\int_{\left\vert \xi\right\vert
>R/h}\varphi\left(  t\right)  |\widehat{\theta\psi_{h}\circ\kappa}\left(
t,\xi\right)  |^{2}d\xi dt\rightarrow0,\qquad\text{as }R\rightarrow
\infty.\label{hoscE}%
\end{equation}
From this it is easy to conclude (\ref{den}) following the lines of the proof
of \cite{GeLei}, Proposition 1.6. In fact, for $a\in C_{c}\left(  V\right)  $,%
\[
\int_{M}a\left(  x\right)  \left\vert e^{i\alpha_{h}ht\Delta/2}u_{h}\left(
x\right)  \right\vert ^{2}dm=\int_{\mathbf{R}}\varphi\left(  t\right)
(\operatorname*{op}\nolimits_{h}\left(  a_{R}\right)  e^{i\alpha_{h}%
ht\Delta/2}u_{h}|e^{i\alpha_{h}ht\Delta/2}u_{h})dt+r\left(  R,h\right)
\]
where $a_{R}\left(  x,\xi\right)  :=a\left(  x\right)  \chi\left(
\xi/R\right)  $ for some $\chi\in C_{c}^{\infty}\left(  \mathbf{R}^{d}\right)
$ with $\chi\left(  0\right)  =1$, $0\leq\chi\leq1$, and $\limsup
_{h\rightarrow0}r\left(  R,h\right)  \rightarrow0$ as $R\rightarrow\infty$
because of (\ref{hoscE}).\medskip
\end{proof}

\begin{proof}
[Proof of Theorem \ref{ThmInvEg}]A direct computation shows:%
\begin{equation}
\frac{d}{dt}(\operatorname*{op}\nolimits_{h}\left(  a\right)  e^{i\alpha
_{h}ht\Delta/2}u_{h}|e^{i\alpha_{h}ht\Delta/2}u_{h})=\frac{i\alpha_{h}h}%
{2}(\left[  \operatorname*{op}\nolimits_{h}\left(  a\right)  ,\Delta\right]
e^{i\alpha_{h}ht\Delta/2}u_{h}|e^{i\alpha_{h}ht\Delta/2}u_{h}). \label{der}%
\end{equation}
Given $\varphi\in C_{c}^{\infty}\left(  \mathbf{R}\right)  $, identities
(\ref{der}) and (\ref{CS}) ensure:%
\begin{equation}
\frac{1}{\alpha_{h}}\int_{\mathbf{R}}\varphi^{\prime}\left(  t\right)
(\operatorname*{op}\nolimits_{h}\left(  a\right)  e^{i\alpha_{h}ht\Delta
/2}u_{h}|e^{i\alpha_{h}ht\Delta/2}u_{h})dt= \label{inv1}%
\end{equation}%
\begin{equation}
=\frac{1}{2}\int_{\mathbf{R}}\varphi\left(  t\right)  (\operatorname*{op}%
\nolimits_{h}\left(  \left\{  a,p\right\}  \right)  e^{i\alpha_{h}ht\Delta
/2}u_{h}|e^{i\alpha_{h}ht\Delta/2}u_{h})dt+\int_{\mathbf{R}}\varphi\left(
t\right)  d_{h}\left(  t\right)  dt, \label{inv2}%
\end{equation}
where $d_{h}\left(  t\right)  \leq Ch\left\Vert u_{h}\right\Vert
_{L^{2}\left(  M\right)  }^{2}$. Taking limits, we conclude that for every
$a\in C_{c}^{\infty}\left(  T^{\ast}M\right)  $ and almost every
$t\in\mathbf{R}$:%
\begin{equation}
\int_{T^{\ast}M}\left\{  a,p\right\}  \left(  x,\xi\right)  \mu\left(
t,dx,d\xi\right)  dt=0, \label{poissonB}%
\end{equation}
and therefore prove i).

Now we turn to the proof of ii). If the symbol $a\in C_{c}^{\infty}\left(
T^{\ast}M\right)  $ is $\phi_{s}$-invariant then $\left\{  a,p\right\}  =0$;
in this case (\ref{CS}) and (\ref{der}) give:%
\begin{align*}
&  (\operatorname*{op}\nolimits_{h}\left(  a\right)  e^{i\alpha_{h}ht\Delta
/2}u_{h}|e^{i\alpha_{h}ht\Delta/2}u_{h})-\left(  \operatorname*{op}%
\nolimits_{h}\left(  a\right)  u_{h}|u_{h}\right) \\
&  =-\alpha_{h}h\frac{i}{2}\int_{0}^{t}(\operatorname*{op}\nolimits_{h}\left(
\left\{  a,r\right\}  \right)  e^{i\alpha_{h}hs\Delta/2}u_{h}|e^{i\alpha
_{h}hs\Delta/2}u_{h})ds-\int_{0}^{t}f_{h}\left(  s\right)  ds,
\end{align*}
with $f_{h}:=2\alpha_{h}h^{-2}\left(  s_{h}e^{i\alpha_{h}ht\Delta/2}%
u_{h}|e^{i\alpha_{h}ht\Delta/2}u_{h}\right)  \leq C\left\Vert u_{h}\right\Vert
_{L^{2}\left(  M\right)  }^{2}o\left(  1\right)  $. Taking imaginary and real
parts, we infer, respectively:%
\begin{equation}
\lim_{h\rightarrow0^{+}}\alpha_{h}h\int_{0}^{t}(\operatorname*{op}%
\nolimits_{h}\left(  \left\{  a,r\right\}  \right)  e^{i\alpha_{h}hs\Delta
/2}u_{h}|e^{i\alpha_{h}hs\Delta/2}u_{h})ds=0, \label{imazero}%
\end{equation}
and, for every $t\in\mathbf{R}$,
\[
\lim_{h\rightarrow0+}(\operatorname*{op}\nolimits_{h}\left(  a\right)
e^{i\alpha_{h}ht\Delta/2}u_{h}|e^{i\alpha_{h}ht\Delta/2}u_{h})=\lim
_{h\rightarrow0+}\left(  \operatorname*{op}\nolimits_{h}\left(  a\right)
u_{h}|u_{h}\right)  ,
\]
which is precisely (\ref{CI}).
\end{proof}

\begin{remark}
Equation (\ref{imazero}) does not give any new information about the
semiclassical measures $\mu\left(  t,\cdot\right)  $. Applying Jacobi's
identity, formula (\ref{defr}), and using the invariance of $a$ we obtain:%
\[
\left\{  a,r\right\}  =\frac{1}{2}\left\{  a,\left\{  p,\log\rho\right\}
\right\}  =-\frac{1}{2}\left\{  p,\left\{  \log\rho,a\right\}  \right\}  .
\]
Therefore (\ref{imazero}) may be restated as $\int_{T^{\ast}M}\left\{
p,\left\{  \log\rho,a\right\}  \right\}  d\mu=0$, which was already deduced
from the invariance property (\ref{poissonB}), since equations (\ref{inv1})
and (\ref{inv2}) imply:%
\[
\alpha_{h}h\int_{\mathbf{R}}\varphi\left(  t\right)  (\operatorname*{op}%
\nolimits_{h}\left(  \left\{  b,p\right\}  \right)  e^{i\alpha_{h}ht\Delta
/2}u_{h}|e^{i\alpha_{h}ht\Delta/2}u_{h})dt=o\left(  1\right)  ,
\]
for every symbol $b\in C_{c}^{\infty}\left(  T^{\ast}M\right)  $.
\end{remark}

\begin{remark}
\label{RmkComm}Note that the restriction $\alpha_{h}=o\left(  h^{-2}\right)  $
may be removed as soon as we have $\left[  \operatorname*{op}\nolimits_{h}%
\left(  a\right)  ,\Delta\right]  =0$ for every invariant $a\in C_{c}^{\infty
}\left(  T^{\ast}M\right)  $. This is the case when $M$ is either the
Euclidean space or the flat torus, for instance.
\end{remark}

\section{\label{SecAv}Averaging formulae}

Now we turn to the proof of Theorem \ref{ThmZoll}. Our first remark concerns
the case in which the average $\left\langle a\right\rangle $ of a symbol is smooth.

\begin{lemma}
\label{Lemma av}Let $\mu$ and $\mu_{0}$ be as in Theorem \ref{ThmE} and
$\alpha_{h}=o\left(  1/h^{2}\right)  $. Suppose that $a\in C_{c}^{\infty
}\left(  T^{\ast}M\right)  $ is such that $\left\langle a\right\rangle $ is
infinitely differentiable in $T^{\ast}M$. Then, for almost every
$t\in\mathbf{R}$,%
\[
\int_{T^{\ast}M}a\left(  x,\xi\right)  \mu\left(  t,dx,d\xi\right)
=\int_{T^{\ast}M}\left\langle a\right\rangle \left(  x,\xi\right)  \mu
_{0}\left(  dx,d\xi\right)  .
\]

\end{lemma}

\begin{proof}
From statement ii) in Theorem \ref{ThmInvEg} we infer, noticing that
$\left\langle a\right\rangle $ is necessarily $\phi_{s}$-invariant, that, for
a.e. $t\in\mathbf{R}$,%
\[
\int_{T^{\ast}M}\left\langle a\right\rangle \left(  x,\xi\right)  \mu\left(
t,dx,d\xi\right)  =\int_{T^{\ast}M}\left\langle a\right\rangle \left(
x,\xi\right)  \mu_{0}\left(  dx,d\xi\right)  .
\]
Now, taking into account that $\mu\left(  t,\cdot\right)  $ is $\phi_{s}%
$-invariant for a.e. $t\in\mathbf{R}$ and using the dominated convergence
theorem, we deduce:%
\begin{align*}
\int_{T^{\ast}M}\left\langle a\right\rangle \left(  x,\xi\right)  \mu\left(
t,dx,d\xi\right)   &  =\lim_{T\rightarrow\infty}\frac{1}{T}\int_{0}^{T}%
\int_{T^{\ast}M}a\left(  \phi_{s}\left(  x,\xi\right)  \right)  \mu\left(
t,dx,d\xi\right)  ds\\
&  =\int_{T^{\ast}M}a\left(  x,\xi\right)  \mu\left(  t,dx,d\xi\right)  ,
\end{align*}
as claimed.\bigskip
\end{proof}

Assume that all the geodesics of $\left(  M,g\right)  $ are closed. This
implies (see \cite{Besse}) that there exists $L>0$ such that, for every
$\left(  x,\xi\right)  \in T^{\ast}M$ with $\left\Vert \xi\right\Vert _{x}=1$,
the geodesic
\[
\mathbf{R}\ni s\longmapsto\phi_{s}\left(  x,\xi\right)  \in T^{\ast}M
\]
is $L$-periodic. As a consequence of homogeneity, the geodesic corresponding
to a general $\left(  x,\xi\right)  \in T^{\ast}M\setminus\left\{  0\right\}
$ is $L/\left\Vert \xi\right\Vert _{x}$-periodic.\bigskip

\begin{proof}
[\textbf{Proof of Theorem \ref{ThmZoll}}]Let $a\in C_{c}^{\infty}\left(
T^{\ast}M\right)  $ vanish in a neighborhood of $\left\{  \xi=0\right\}  $.
Due to the periodicity of the geodesic flow, the average of $a$ equals:%
\[
\left\langle a\right\rangle \left(  x,\xi\right)  :=\frac{\left\Vert
\xi\right\Vert _{x}}{L}\int_{0}^{L/\left\Vert \xi\right\Vert _{x}}a\left(
\phi_{s}\left(  x,\xi\right)  \right)  ds,
\]
for every $\left(  x,\xi\right)  \in T^{\ast}M$. It follows that $\left\langle
a\right\rangle $ is a smooth function; using Lemma \ref{Lemma av} we conclude
that identity (\ref{averZ}) holds for $a$. This implies that $\mu\left(
t,T^{\ast}M\setminus\left\{  0\right\}  \right)  =\mu_{0}\left(  T^{\ast
}M\setminus\left\{  0\right\}  \right)  $ for a.e. $t\in\mathbf{R}$. Since $M$
is compact, (\ref{den}) implies that, again for a.e. $t\in\mathbf{R}$, the
total masses of $\mu\left(  t,\cdot\right)  $ and $\mu_{0}$ are equal.
Finally, as $\mu_{0}\left(  \left\{  \xi=0\right\}  \right)  =0$ we must
necessarily have $\mu\left(  t,\left\{  \xi=0\right\}  \right)  =0$ and
formula (\ref{averZ}) follows for arbitrary $a\in C_{c}^{\infty}\left(
T^{\ast}M\right)  $.\medskip
\end{proof}

\begin{proof}
[Proof of Proposition \ref{PropRd}]The proof is immediate: for almost every
$t\in\mathbf{R}$ the measures $\mu\left(  t,\cdot\right)  $ are invariant by
translations $\left(  x,\xi\right)  \mapsto\left(  x+s\xi,\xi\right)  $ (by
Theorem \ref{ThmInvEg}, i)) and do not charge the set $\left\{  \xi=0\right\}
$, as the projection of $\mu\left(  t,\cdot\right)  $ on $\xi$ coincides with
that of $\mu_{0}$ (this can be checked directly, or seen as a consequence of
Theorem \ref{ThmInvEg}, ii)). This and the fact that $\mu\left(
t,\cdot\right)  $ is finite for a.e. $t$ forces $\mu=0$.\medskip
\end{proof}

\begin{proof}
[Proof of Proposition \ref{Prop AvTorus}]Let $a\in C_{c}^{\infty}\left(
T^{\ast}\mathbf{T}^{d}\right)  $, and consider its average $\left\langle
a\right\rangle $ along the geodesic flow. The hypothesis $\mu_{0}\left(
\mathbf{T}^{d}\times\Omega\right)  =0$ ensures that, for $\mu_{0}$-almost
every $\xi\in\mathbf{R}^{d}$ we have
\[
\left\langle a\right\rangle \left(  x,\xi\right)  =\overline{a}\left(
\xi\right)  :=\frac{1}{\left(  2\pi\right)  ^{d}}\int_{\mathbf{T}^{d}}a\left(
y,\xi\right)  dy,
\]
as only dense geodesics are involved in the average. We cannot apply Lemma
\ref{Lemma av} in this setting, since $\left\langle a\right\rangle $ is not
smooth. However, by Theorem \ref{ThmInvEg}, ii) (note that there is no
restriction on $\alpha_{h}$, by Remark \ref{RmkComm}), we have that
$\int_{\mathbf{T}^{d}}\mu\left(  t,dx,\cdot\right)  =\int_{\mathbf{T}^{d}}%
\mu_{0}\left(  dx,\cdot\right)  $ and therefore, for \emph{a.e.}
$t\in\mathbf{R}$,%
\[
\int_{T^{\ast}\mathbf{T}^{d}}\left\langle a\right\rangle \left(  x,\xi\right)
\mu\left(  t,dx,d\xi\right)  =\int_{T^{\ast}\mathbf{T}^{d}}\overline{a}\left(
\xi\right)  \mu_{0}\left(  dx,d\xi\right)  .
\]
\ We apply the dominated convergence theorem and use the invariance of
$\mu\left(  t,\cdot\right)  $ under the geodesic flow to conclude%
\[
\int_{T^{\ast}\mathbf{T}^{d}}a\left(  x,\xi\right)  \mu\left(  t,dx,d\xi
\right)  =\int_{T^{\ast}\mathbf{T}^{d}}\left\langle a\right\rangle \left(
x,\xi\right)  \mu\left(  t,dx,d\xi\right)  ,
\]
for \emph{a.e.} $t\in\mathbf{R}$, and the proof follows.\medskip
\end{proof}

\section{\label{SecT}Concentration on resonant frequencies}

In this section we prove Proposition \ref{PropQuantTorus}. From now on, we
shall identify functions defined on $\mathbf{T}^{d}$ to the $2\pi
\mathbf{Z}^{d}$-periodic functions defined on $\mathbf{R}^{d}$. If so, the
Euclidean Wigner distribution of
\[
u\left(  x\right)  =\sum_{k\in\mathbf{Z}^{d}}\widehat{u}\left(  k\right)
\frac{e^{ik\cdot x}}{\left(  2\pi\right)  ^{d/2}}\in L^{2}\left(
\mathbf{T}^{d}\right)
\]
is given by:%
\[
l_{u}^{h}\left(  x,\xi\right)  :=\sum_{k,j\in\mathbf{Z}^{d}}\widehat{u}\left(
k\right)  \overline{\widehat{u}\left(  j\right)  }\frac{e^{i\left(
k-j\right)  \cdot x}}{\left(  2\pi\right)  ^{d}}\delta_{\frac{h}{2}\left(
k+j\right)  }\left(  \xi\right)  .
\]
It is easy to check that $l_{u}^{h}$ differs from the Wigner distribution
$w_{u}^{h}$ defined in Section \ref{SecSSS} by an $\mathcal{O}\left(
h\right)  $ term. Therefore, their limits coincide and give the usual
semiclassical measures. Clearly, $L_{u}^{h}\left(  t,\cdot\right)
:=l_{e^{i\alpha_{h}ht\Delta/2}u_{h}}^{h}$ satisfies,
\[
\int_{\mathbf{R}}\varphi\left(  t\right)  \left\langle L_{u}^{h}\left(
t,\cdot\right)  ,a\right\rangle dt=\frac{1}{\left(  2\pi\right)  ^{d/2}}%
\sum_{k,j\in\mathbf{Z}^{d}}\widehat{\varphi}\left(  \frac{\left\vert
k\right\vert ^{2}-\left\vert j\right\vert ^{2}}{2}\right)  \widehat{u}\left(
k\right)  \overline{\widehat{u}\left(  j\right)  }a_{j-k}\left(  \frac
{hk+hj}{2}\right)  ,
\]
for every $a\in C_{c}^{\infty}\left(  T^{\ast}\mathbf{T}^{d}\right)  $ of the
form $a\left(  x,\xi\right)  =\left(  2\pi\right)  ^{-d/2}\sum_{k\in
\mathbf{Z}^{d}}a_{k}\left(  \xi\right)  e^{ik\cdot x}$.\medskip

We now define the sequences $\left(  u_{h}\right)  $ and $\left(
v_{h}\right)  $. Let $\theta_{0}\in\mathbf{R}^{d}\setminus\Omega$; let
$\left(  k_{n}\right)  $ be a sequence in $\mathbf{Q}^{d}$ such that
$\lim_{n\rightarrow\infty}k_{n}=\theta_{0}$. Suppose that $k_{n}=\left(
p_{n}^{1}/q_{n}^{1},...,p_{n}^{d}/q_{n}^{d}\right)  $ with $p_{n}^{j}$ and
$q_{n}^{j}$ relatively prime; let $q_{n}$ denote the least common multiple of
$q_{n}^{1}$,...,$q_{n}^{d}$ and write $\lambda_{1}:=q_{1}$, and $\lambda
_{n}=q_{n}\lambda_{n-1}$ for $n>1$. Clearly, only a finite number of the
$q_{n}$ may be equal to one, therefore $\lim_{n\rightarrow\infty}\lambda
_{n}=\infty$. Finally, set $h_{n}:=1/\left(  \lambda_{n}\right)  ^{2}$.

Now write,%
\[
S^{1}\left(  x\right)  :=\xi_{0}\cdot x,\qquad S_{n}^{2}\left(  x\right)
:=\xi_{0}\cdot x+\sqrt{h_{n}}k_{n}\cdot x,
\]
and%
\[
u_{h_{n}}\left(  x\right)  :=\varrho\left(  x\right)  e^{iS^{1}\left(
x\right)  /h_{n}},\qquad v_{h_{n}}\left(  x\right)  :=\varrho\left(  x\right)
e^{iS_{n}^{2}\left(  x\right)  /h_{n}}.
\]
Since $\lambda_{n}^{2}\xi_{0},\lambda_{n}k_{n}\in\mathbf{Z}^{d}$, the Fourier
coefficients of $u_{h_{n}}$ and $v_{h_{n}}$ are obtained from those of
$\varrho$ as:%
\[
\widehat{u_{h_{n}}}\left(  k\right)  =\widehat{\varrho}\left(  k-\lambda
_{n}^{2}\xi_{0}\right)  ,\qquad\widehat{v_{h_{n}}}\left(  k\right)
=\widehat{\varrho}\left(  k-\lambda_{n}^{2}\xi_{0}-\lambda_{n}k_{n}\right)  .
\]
The proof of the fact that the limits of $\left(  l_{u_{h_{n}}}^{h_{n}%
}\right)  ,\left(  l_{v_{h_{n}}}^{h_{n}}\right)  $ coincide with the measure
given by (\ref{mu0osc}) is simple and may be reconstructed following the same
lines as that for the evolution case. We therefore concentrate on the
latter.\medskip

Let us now compute the limit of $\left(  L_{u_{h_{n}}}^{h_{n}}\right)  $;
clearly, it suffices to consider the limit against test functions
$a_{l}\left(  x,\xi\right)  :=b\left(  \xi\right)  e^{-il\cdot x}$ with $b\in
C_{c}^{\infty}\left(  \mathbf{R}^{d}\right)  $ and $\varphi\in L^{1}\left(
\mathbf{R}\right)  $ with $\widehat{\varphi}\in C_{c}\left(  \mathbf{R}%
\right)  $. We can write:%
\begin{align*}
\int_{\mathbf{R}}\varphi\left(  t\right)  \left\langle L_{u_{h_{n}}}^{h_{n}%
}\left(  t,\cdot\right)  ,a_{l}\right\rangle dt  &  =\sum_{k-j=l}b\left(
\frac{h_{n}k+h_{n}j}{2}\right)  \widehat{\varphi}\left(  \frac{\left\vert
k\right\vert ^{2}-\left\vert j\right\vert ^{2}}{2}\right)  \widehat{u_{h_{n}}%
}\left(  k\right)  \overline{\widehat{u_{h_{n}}}\left(  j\right)  }\\
&  =\sum_{k-j=l}b\left(  h_{n}\frac{k+j}{2}+\xi_{0}\right)  \widehat{\varphi
}\left(  l\cdot\left(  \frac{k+j}{2}+\lambda_{n}^{2}\xi_{0}\right)  \right)
\widehat{\varrho}\left(  k\right)  \overline{\widehat{\varrho}\left(
j\right)  }.
\end{align*}
If $l\cdot\xi_{0}\not =0$ then the expression above vanishes as $n\rightarrow
\infty$. To see this, suppose that $\operatorname*{supp}\widehat{\varphi
}\subset\left(  -R,R\right)  $; clearly:%
\begin{equation}
\left\vert \int_{\mathbf{R}}\varphi\left(  t\right)  \left\langle L_{u_{h_{n}%
}}^{h_{n}}\left(  t,\cdot\right)  ,a_{l}\right\rangle dt\right\vert
\leq\left\Vert b\right\Vert _{L^{\infty}\left(  \mathbf{R}^{d}\right)
}\left\vert \sum_{j\in\mathbf{Z}^{d}}\widehat{\varphi}\left(  l\cdot\left(
j+\frac{l}{2}+\lambda_{n}^{2}\xi_{0}\right)  \right)  \widehat{\varrho}\left(
j+l\right)  \overline{\widehat{\varrho}\left(  j\right)  }\right\vert .
\label{nonzeromode}%
\end{equation}
The distance $d_{n}$ between the hyperplane $l\cdot\left(  \xi+l/2+\lambda
_{n}^{2}\xi_{0}\right)  =0$ and the origin tends to infinity as $n\rightarrow
\infty$. Therefore, for $n$ large enough we can estimate (\ref{nonzeromode})
by
\[
\left\Vert b\right\Vert _{L^{\infty}\left(  \mathbf{R}^{d}\right)  }\left\Vert
\widehat{\varphi}\right\Vert _{L^{\infty}\left(  \mathbf{R}\right)  }%
\sum_{\left\vert j\right\vert >d_{n}-2R}\left\vert \widehat{\varrho}\left(
j+l\right)  \overline{\widehat{\varrho}\left(  j\right)  }\right\vert ,
\]
which tends to zero as $n\rightarrow\infty$ since $\varrho\in L^{2}\left(
\mathbf{T}^{d}\right)  $.

When $l\cdot\xi_{0}=0$ we have:%
\[
\int_{\mathbf{R}}\varphi\left(  t\right)  \left\langle L_{u_{h_{n}}}^{h_{n}%
}\left(  t,\cdot\right)  ,a_{l}\right\rangle dt=\sum_{k-j=l}b\left(
h_{n}\frac{k+j}{2}+\xi_{0}\right)  \widehat{\varphi}\left(  \frac{\left\vert
k\right\vert ^{2}-\left\vert j\right\vert ^{2}}{2}\right)  \widehat{\varrho
}\left(  k\right)  \overline{\widehat{\varrho}\left(  j\right)  },
\]
and letting $n\rightarrow\infty$ gives (recall that $\widehat{\varphi}$ is
compactly supported):%
\[
\int_{\mathbf{R\times T}^{d}}\varphi\left(  t\right)  a_{l}\left(  \xi\right)
\mu_{\left(  u_{h}\right)  }\left(  t,dx,d\xi\right)  =b\left(  \xi
_{0}\right)  \sum_{k-j=l}\widehat{\varphi}\left(  \frac{\left\vert
k\right\vert ^{2}-\left\vert j\right\vert ^{2}}{2}\right)  \widehat{\varrho
}\left(  k\right)  \overline{\widehat{\varrho}\left(  j\right)  }.
\]
In conclusion, for a general $a\in C_{c}^{\infty}\left(  T^{\ast}%
\mathbf{T}^{d}\right)  $ of the form $a\left(  x,\xi\right)  :=\sum
_{l\in\mathbb{Z}^{d}}b_{l}\left(  \xi\right)  e^{-il\cdot x}$ one has:%
\begin{align*}
\int_{\mathbf{R\times T}^{d}}\varphi\left(  t\right)  a\left(  x,\xi\right)
\mu_{\left(  u_{h}\right)  }\left(  t,dx,d\xi\right)   &  =\sum_{l\cdot\xi
_{0}=0}\sum_{k-j=l}b_{l}\left(  \xi_{0}\right)  \widehat{\varphi}\left(
\frac{\left\vert k\right\vert ^{2}-\left\vert j\right\vert ^{2}}{2}\right)
\widehat{\varrho}\left(  k\right)  \overline{\widehat{\varrho}\left(
j\right)  }\\
&  =\int_{\mathbf{T}^{d}}\left\langle a\right\rangle \left(  x,\xi_{0}\right)
\left\vert e^{it\Delta/2}\rho\left(  x\right)  \right\vert ^{2}dx,
\end{align*}
($\left\langle a\right\rangle $ being defined by (\ref{aver})) and therefore
(\ref{mu1}) holds for $\left(  u_{h_{n}}\right)  $. \medskip

Now we turn to the corresponding computation for$\left(  L_{v_{h_{n}}}^{h_{n}%
}\right)  $. Reasoning as before, we have:%
\[
\left\vert \int_{\mathbf{R}}\varphi\left(  t\right)  \left\langle L_{v_{h_{n}%
}}^{h_{n}}\left(  t,\cdot\right)  ,a_{l}\right\rangle dt\right\vert
\leq\left\Vert b\right\Vert _{L^{\infty}\left(  \mathbf{R}^{d}\right)
}\left\vert \sum_{j\in\mathbf{Z}^{d}}\widehat{\varphi}\left(  l\cdot\left(
j+\frac{l}{2}+\lambda_{n}^{2}\xi_{0}+\lambda_{n}k_{n}\right)  \right)
\widehat{\varrho}\left(  j+l\right)  \overline{\widehat{\varrho}\left(
j\right)  }\right\vert .
\]
Now, if $l\neq0$ it is easy to check that distance between the hyperplane
$l\cdot\left(  \xi+l/2+\lambda_{n}^{2}\xi_{0}+\lambda_{n}k_{n}\right)  =0$ and
the origin always tends to infinity, since $\lim_{n\rightarrow\infty}l\cdot
k_{n}=l\cdot\theta_{0}\neq0$ and therefore, $\lim_{n\rightarrow\infty
}\left\vert l\cdot\left(  l/2+\lambda_{n}^{2}\xi_{0}+\lambda_{n}k_{n}\right)
\right\vert =\infty$. The same argument we used for $\left(  u_{h_{n}}\right)
$ now gives us:%
\[
\lim_{n\rightarrow\infty}\int_{\mathbf{R}}\varphi\left(  t\right)
\left\langle L_{v_{h_{n}}}^{h_{n}}\left(  t,\cdot\right)  ,a_{l}\right\rangle
dt=0.
\]
When $l=0$ we have%
\[
\int_{\mathbf{R}}\varphi\left(  t\right)  \left\langle L_{u_{h_{n}}}^{h_{n}%
}\left(  t,\cdot\right)  ,a_{l}\right\rangle dt=\widehat{\varphi}\left(
0\right)  \sum_{j\in\mathbf{Z}^{d}}b\left(  h_{n}j+\sqrt{h_{n}}k_{n}+\xi
_{0}\right)  \left\vert \widehat{\varrho}\left(  j\right)  \right\vert ^{2},
\]
which converges precisely to $\left(  2\pi\right)  ^{-d}\widehat{\varphi
}\left(  0\right)  b\left(  \xi_{0}\right)  \left\Vert \varrho\right\Vert
_{L^{2}\left(  \mathbf{T}^{d}\right)  }^{2}$. This shows that (\ref{mu2}) holds.

\section{Schr\"{o}dinger equations with a potential\label{SecV}}

Some of the results presented here have an analogue for the more general
Schr\"{o}dinger equation:
\begin{equation}
ih\partial_{t}\psi_{h}\left(  t,x\right)  +\frac{h^{2}}{2}\Delta\psi
_{h}\left(  t,x\right)  -V\left(  x\right)  \psi_{h}\left(  t,x\right)
=0\qquad\left(  t,x\right)  \in\mathbf{R}\times M. \label{SEV}%
\end{equation}
provided we assume that the potential $V\in C^{2}\left(  M\right)  $
satisfies:%
\begin{equation}
\text{the Hamiltonian flow }\phi_{t}^{H}\text{ on }T^{\ast}M\text{ associated
to }H\left(  x,\xi\right)  :=\frac{1}{2}\left\Vert \xi\right\Vert _{x}%
^{2}+V\left(  x\right)  \text{ is complete;} \label{VI}%
\end{equation}%
\begin{equation}
\text{the operator }\mathcal{H}_{h}:=\frac{h^{2}}{2}\Delta-V\text{ is
essentially self-adjoint in }L^{2}\left(  M\right)  . \label{Hh}%
\end{equation}
Note that both conditions are met when, for instance, $V\geq-C$ for some
$C>0$. See \cite{ShubinSelfAdj, ReedSimon2} and the references therein for a
thorough discussion on this issue.

The analogues of Theorem \ref{ThmE} and \ref{ThmInvEg} hold for the evolved
Wigner distributions in this framework:
\[
\left\langle W_{u_{h}}^{h}\left(  t,\cdot\right)  ,a\right\rangle
:=(\operatorname*{op}\nolimits_{h}\left(  a\right)  e^{it/h\mathcal{H}_{h}%
}u_{h}|e^{it/h\mathcal{H}_{h}}u_{h}).
\]

\begin{theorem}
Let $\left(  u_{h}\right)  $ be a sequence bounded in $L^{2}\left(  M\right)
$ satisfying conditions (\ref{h-osc}) (with $h^{2}\Delta$ replaced by
$\mathcal{H}_{h}$) and (\ref{convID}). Let $\mu_{0}$ be its semiclassical
measure. Then, at least for some subsequence, the following hold.\medskip

i) There exists a measure $\mu\in L^{\infty}\left(  \mathbf{R};\mathcal{M}%
_{+}\left(  T^{\ast}M\right)  \right)  $ such that%
\[
\lim_{h\rightarrow0^{+}}\int_{\mathbf{R}}\varphi\left(  t\right)  \left\langle
W_{u_{h}}^{h}\left(  \alpha_{h}t,\cdot\right)  ,a\right\rangle dt=\int
_{\mathbf{R}}\varphi\left(  t\right)  \int_{T^{\ast}M}a\left(  x,\xi\right)
\mu\left(  t,dx,d\xi\right)  dt,
\]
for every $\varphi\in L^{1}\left(  \mathbf{R}\right)  $ and $a\in
C_{c}^{\infty}\left(  T^{\ast}M\right)  $.\medskip

ii) For a.e.$t\in\mathbf{R}$, the measure $\mu\left(  t,\cdot\right)  $ is
invariant under the Hamiltonian flow $\phi_{s}^{H}$. \medskip

iii) Given $a\in C_{c}^{\infty}\left(  T^{\ast}M\right)  $ invariant under the
classical flow $\phi_{t}^{H}$ and $\alpha_{h}=o\left(  1/h^{2}\right)  $, the
following holds:%
\[
\lim_{h\rightarrow0^{+}}\left\langle W_{u_{h}}^{h}\left(  \alpha_{h}%
t,\cdot\right)  ,a\right\rangle =\int_{T^{\ast}M}a\left(  x,\xi\right)
\mu_{0}\left(  dx,d\xi\right)  ,\qquad\text{\emph{for every}}\emph{\ }%
t\in\mathbf{R.}%
\]

\end{theorem}

This is a consequence of the fact that the presence of the potential $V$ does
not introduce terms of order $h^{2}$ in the expansion for the commutator:%
\[
\left[  \operatorname*{op}\nolimits_{h}\left(  a\right)  ,\mathcal{H}%
_{h}\right]  =\frac{h}{i}\operatorname*{op}\nolimits_{h}\left(  \left\{
a,H\right\}  \right)  +h^{2}\operatorname*{op}\nolimits_{h}\left(  \left\{
a,r\right\}  \right)  +\mathcal{O}\left(  h^{3}\right)  .
\]
It is easy to prove using Lemma \ref{Lemma av} the following analogue of
Theorem \ref{ThmZoll} in this setting.

\begin{theorem}
\label{ThmPeriodic}Suppose that the Hamiltonian flow $\phi_{t}^{H}$ is
periodic and $\alpha_{h}=o\left(  1/h^{2}\right)  $. Let $\mu_{0}$ be the
semiclassical measure of some sequence $\left(  u_{h}\right)  $ in
$L^{2}\left(  M\right)  $ satisfying (\ref{h-osc}) and (\ref{convID}). Then,
for any subsequence for which (\ref{conv}) holds we have the averaging
formula:%
\begin{equation}
\lim_{h\rightarrow0^{+}}\int_{\mathbf{R}}\varphi\left(  t\right)  \left\langle
W_{u_{h}}^{h}\left(  \alpha_{h}t,\cdot\right)  ,a\right\rangle dt=\left(
\int_{\mathbf{R}}\varphi\left(  t\right)  dt\right)  \int_{T^{\ast}%
M}\left\langle a\right\rangle \left(  x,\xi\right)  \mu_{0}\left(
dx,d\xi\right)  , \label{formZ}%
\end{equation}
for every $a\in C_{c}^{\infty}\left(  T^{\ast}M\right)  $ and $\varphi\in
L^{1}\left(  \mathbf{R}\right)  $, the average $\left\langle a\right\rangle $
being taking with respect to $\phi_{t}^{H}$.
\end{theorem}

\begin{remark}
If $\phi_{t}^{H}$ is just periodic in $X:=H^{-1}\left(  E_{1},E_{2}\right)  $
for some $E_{1}<E_{2}$, then formula (\ref{formZ}) holds for functions $a\in
C_{c}^{\infty}\left(  X\right)  $.
\end{remark}

\begin{remark}
\label{RmkV}The conclusions of Theorem \ref{ThmZoll} and Proposition
\ref{PropRd} also hold for the solutions to the adimensional equation:
\[
i\partial_{t}v_{h}+\frac{1}{2}\Delta v_{h}-Vv_{h}=0,
\]
as they can be written as solutions the semiclassical equation (\ref{SEV})
with potential $h^{2}V$ evaluated at time $t/h$. Therefore, as an immediate
consequence of the proof, the conclusions of Theorem \ref{ThmInvEg} hold with
$\phi_{s}$ being the geodesic flow of $\left(  M,g\right)  $.
\end{remark}

\noindent\textbf{Acknowledgments. }The author wishes to thank Patrick
G\'{e}rard for having introduced him to this problem and for sharing with him
many interesting ideas. This work was initiated as the author was visiting the
\emph{Laboratoire de Mathematiques} at \emph{Universit\'{e} de Paris-Sud}. He
wishes to thank this institution for its kind hospitality. He also thanks the
anonimous referees for their comments and suggestions, which have considerably
increased the quality of the final version of this article.

\end{document}